\documentclass[11pt]{amsart}
\usepackage{amsmath}
\usepackage[dvips]{graphicx}
\usepackage{amsfonts}
\usepackage{amsopn}
\usepackage{amsthm} 
\usepackage{amssymb}
\usepackage{color}
\usepackage{mathrsfs}
\usepackage{hyperref}

\usepackage{latexsym}
\usepackage{amsgen}
\usepackage{comment}

\usepackage[title]{appendix}
\usepackage{algorithm}
\usepackage{algpseudocode}

\usepackage{booktabs, multirow} 
\usepackage{soul}
\usepackage[table]{xcolor} 
\usepackage{changepage,threeparttable} 

\DeclareMathOperator*{\argmax}{arg\,max}
\DeclareMathOperator*{\argmin}{arg\,min}
\newtheorem{theorem}{Theorem}[section]

\newtheorem{proposition}[theorem]{Proposition}

\numberwithin{equation}{section}

\renewcommand{\u}{{\textbf{u}}}

\newcommand{\be}{\begin{equation}}
\newcommand{\ee}{\end{equation}}

\newcommand{\ba}{\begin{array}}
	\newcommand{\ea}{\end{array}}

\usepackage{amssymb}

\numberwithin{equation}{section}

\title[]{A primal-dual approach for solving conservation Laws with Implicit in Time Approximations}

\author[Liu]{Siting Liu}
\thanks{Department of Mathematics, University of California, Los Angeles, CA 90095, USA. (siting6@math.ucla.edu, sjo@math.ucla.edu)}
\author[Osher]{Stanley Osher}
\author[Li]{Wuchen Li}
\thanks{Department of Mathematics, University of South Carolina, Columbia, SC 29208, USA. (wuchen@mailbox.sc.edu)}
\author[Shu]{Chi-Wang Shu}
\thanks{Division of Applied Mathematics, Brown University, Providence,
RI 02912, USA. (chi-wang\_shu@brown.edu)}
\thanks{S. Liu and S. Osher thank the funding from AFOSR MURI FA9550-18-1-0502 and ONR: N00014-20-1-2093 and N00014-20-1-2787. W. Li thanks the funding from NSG RTG2038080 and AFOSR FA9550-18-1-0502. C. Shu thanks the funding from AFOSR FA9550-20-1-0055 and NSF DMS-2010107.}

\begin{document} 
	\maketitle

	\begin{abstract}
    In this work, we propose a novel framework for the numerical solution of time-dependent conservation laws with implicit schemes via primal-dual hybrid gradient methods. 
    We solve an initial value problem (IVP) for the partial differential equation (PDE) by casting it as a saddle point of a min-max problem and using iterative optimization methods to find the saddle point.
    Our approach is flexible with the choice of both time and spatial discretization schemes.
    It benefits from the implicit structure and gains large regions of
    stability, and overcomes the restriction on the
mesh size in time by explicit schemes from Courant--Friedrichs--Lewy (CFL) conditions (really via von Neumann stability analysis).
    Nevertheless, it is highly parallelizable and easy-to-implement. 
    In particular, no nonlinear inversions are required!
    Specifically, we illustrate our approach using the finite difference scheme and discontinuous Galerkin method for the spatial scheme; backward Euler and backward differentiation formulas for implicit discretization in time. 
    Numerical experiments illustrate the effectiveness and robustness of the approach.
    In future work, we will demonstrate that our idea of replacing an initial-value evolution equation with this primal-dual hybrid gradient approach has great advantages in many other situations.
	\end{abstract}

\section{Introduction}	

High order numerical approximations of hyperbolic systems of conservation laws and their viscous regularizations have been well studied for decades. 
High order accurate schemes in the spatial domain, such as discontinuous Galerkin methods (DG methods)\cite{cockburn2001runge}, essentially non-oscillatory (ENO)\cite{harten1987uniformly} methods, and weighted essentially non-oscillatory (WENO) schemes \cite{liu1994weighted}, have been quite successful in computational fluid dynamics, magnetohydrodynamics and numerous other fields.
In particular, the successful methods suppress spurious oscillations, which tend to occur when the solution develops discontinuities. 
As for high order schemes in the time domain, explicit methods such as the forward Euler method and Runge-Kutta method have 
restrictions on the length of the time step from Courant-Friedrichs-Lewy (CFL) conditions (von Neumann stability analysis).
While explicit methods sometimes require impractically small time steps for many stiff problems, implicit methods are used with larger time steps.
However, implicit methods typically require extra computations and can be hard to implement.
Newton-type solvers take ample storage of the Jacobian matrix and need a good initial guess for fast convergence.

In this paper, we propose a novel computational framework for solving implicit numerical PDEs. It is easy to implement and has the flexibility to accommodate different high order numerical schemes. We design a primal-dual approach to solve the implicit schemes of conservation laws efficiently.
The key technique here is introducing a Lagrange multiplier, transforming the initial value problem into a min-max problem, and applying the first-order optimization method to find the saddle point.
We use the primal-dual hybrid gradient descent method (PDHG)~\cite{champock11,champock16} with a particular choice of preconditioner to solve the min-max problem, of which the saddle point corresponds to the solution of the initial value problem.
The saddle point structure leads to a forward-backward coupled system of equations, where the original conservation law equation (the primal equation) runs forward in time, and the dual equation (the Lagrange multiplier) solves an equation with a specified terminal time condition.
As for numerical approximations, we discretize the conservation law with the implicit scheme (such as the backward Euler method). Then with a summation by parts, we get a forward Euler scheme for the dual equation.
This equation also inherits an implicit scheme since its terminal time condition is given.
We alternatively update the solutions for the primal and dual equations in the spatial-time domain in parallel by applying the proximal gradient descent (ascent) method.
Moreover, our primal-dual framework can be generalized to various high order schemes in time. This paper discusses the finite difference scheme and discontinuous Galerkin method.

In the literature, optimization has a close connection with numerical PDEs and optimal control problems. 
The celebrated Benamou-Brenier formulation\cite{benamou2000computational} provides a computational fluid mechanics perspective of the optimal transport. It casts the problem as a convex minimization problem and solves it numerically via an augmented Lagrangian method. 
The paper \cite{carrillo2022primal} proposes a new computational method by leveraging the underlying variational structure of the PDEs. Their methods first discretize the problem in time use the  Jordan-- Kinderlehrer--Otto (JKO) type scheme and compute discrete time equations using Benamou-Brenier formula. Recent work \cite{cheng2022new1} introduces a space-time Lagrange multiplier to enforce the positivity of the solution and construct numerical schemes via predictor-corrector approach. In\cite{li2021controlling,li2022controlling}, a novel control problem is proposed on a modified optimal transport space based on conservation laws with diffusion regularization. The computations of such control problems utilize the forward-backward coupled structure of the continuity equations and the value functions. The paper also discusses calculation of a degenerate case, where no control is enforced, and gives a solution to an initial value problem. In  
\cite{LI2022111409}, a control problem associated with nonlinear reaction diffusion equations is studied and primal-dual methods 
are used for the computation.
In \cite{zang2020weak}, they propose a weak adversarial networks approach, in which they study a saddle point problem with nonlinear dependence on the dual (test) variable. They parameterize both the solution and its dual (test) functions with deep neural networks for solving PDEs.
Compared with the works above, we study a simple inf-sup saddle point problem which depends linearly on the dual variable. The saddle point solves conservation laws directly.
The PDHG algorithm provides a provable contraction for solving the proposed saddle point problem when the IVP is a linear PDE.
The parallelizable property (in space-time) of our method is based on the iterative proximal updating of the primal and dual variables. This is different from the parareal (parallel-in-time) method \cite{gander2007analysis}, which focuses on multiple shooting along the time axis. In addition, the PDHG methods provide a simple to implement algorithm,  compared with the Newton's method used in parareal method. Our method does not compute the inversion of Jacobian matrix for the implicit schemes.

The rest of the paper is organized as follows. Section \ref{sec:continuous} describes how the primal-dual approach solves an initial value problem at the continuous level. Next, we discuss in Section~\ref{sec:numerics} the implicit numerical scheme and 
the optimization method we use to solve the min-max problem. In particular, we point out that our approach works as an iterative solver in general for solving one-timestep forward of an initial value problem. We present several numerical examples in Section~\ref{sec:example} to validate and demonstrate the effectiveness of our framework. We give some concluding remarks in Section~\ref{sec:summary}. 

\section{A inf-sup problem formulation}\label{sec:continuous}
In this section, we derive the inf-sup formulation of PDEs, where the saddle point corresponds to the solution of the initial value problem. We review the primal-dual hybrid gradient descent method and its variations and discuss the application of the algorithm in our setup.
\subsection{From the initial value problem to a saddle point formulation}
We consider the following initial value problem of scalar conservation law defined over the domain $\Omega \times[0,T]$:
\begin{equation}
\label{eq:conservation_law_gamma}
\begin{split}
&        \partial_t u(x,t) + \partial_x f(u(x,t)) - \partial_{x} \left(\gamma(x) \partial_x u\left(x,t\right) \right) = 0,\\
& u(x,0) = u_0(x).
\end{split}
\end{equation}
For simplicity, we assume $\Omega \subset \mathbb{R}$ satisfies periodic boundary conditions. Here $f(u)$ is the flux term, and $\gamma = 0$ for the nonviscous case and $\gamma(x)>0$ for viscous conservation laws.
We seek for a function $u: \Omega \times[0,T] \rightarrow \mathbb{R}$ that satisfies Equation \eqref{eq:conservation_law_gamma}. It is straightforward to come up with the following minimization problem: 

\begin{align}
     \label{eq:min_dirac_delta}
      &\min_{u} \mathbf{1}_{u\in \mathcal{U}}.
      \end{align}
      \begin{align*}
      &\quad \mathcal{U}=\{u:\partial_t u(x,t) +\partial_x f(u(x,t)) - \partial_{x} \left(\gamma(x) \partial_x u\left(x,t\right) \right)  = 0 \;\text{for all}\; (x,t) \in \Omega \times[0,T],\\ & u(x,0) = u_0(x)\},          
      \end{align*}

where the minimizer is the solution to the initial value problem.
Here, we have the indicator function defined as follows 
\begin{align*}
    \mathbf{1}_{a\in \mathcal{A}} = \begin{cases}
    0 \quad \text{if} \; a \in \mathcal{A} \\
    \infty \quad  \text{if} \; a \notin \mathcal{A}. 
    \end{cases}
\end{align*}
The optimization problem \eqref{eq:min_dirac_delta} can also be rewritten as follows:
\begin{align}
\label{eqn:min_constraint}
    \min_{u\in \mathcal{U}} \mathbf{0},
\end{align}
where the objective function is a constant, while the continuity equation lies in the constraint. Solving the constrained optimization problem in \eqref{eqn:min_constraint} is equivalent to find the minimum in problem \eqref{eq:min_dirac_delta}.

The Lagrange multiplier technique is a popular approach to tackle constrained optimization problems. 
By introducing a Lagrange multiplier $\phi: \Omega \times[0,T] \rightarrow \mathbb{R}$, we can remove the constraint from problem \eqref{eqn:min_constraint}, which leads to a min-max problem:
\begin{equation}\label{eqn:min_max}
    \min_{u,u(\cdot,0) = u_0} \max_{\phi} \mathcal{L}(u,\phi),
\end{equation}
where
\begin{align*}
  \mathcal{L}(u,\phi) & = \int_{0}^T \int_{\Omega} \phi(x,t) \left( \partial_t u(x,t) + \partial_x f(u(x,t)) - \partial_{x} \left(\gamma(x) \partial_x u\left(x,t\right) \right) \right)dx dt \\
   & = -\int_{0}^T \int_{\Omega} u(x,t) \left( \partial_t \phi(x,t)  + \partial_{x} \left(\gamma(x) \partial_x \phi\left(x,t\right) \right) \right) + \partial_x \phi(x,t) f(u(x,t)) dx dt \\
   &\quad + \int_{\Omega} u(x,T)\phi(x,T) - u(x,0) \phi(x,0)dx.
\end{align*}
In the above formula, we apply integration by parts in spatial and time domain respectively. 
By taking the first-order optimality condition, we obtain the following system of equations
\begin{align}
    \begin{cases}
     \partial_t u(x,t) + \partial_x f(u(x,t)) - \partial_{x} \left(\gamma(x) \partial_x u\left(x,t\right) \right) = 0, \\
     u(x,0) = u_0(x),\\
     \partial_t \phi(x,t) + \partial_x \phi(x,t) f'(u(x,t)) + \partial_{x} \left(\gamma(x) \partial_x \phi\left(x,t\right) \right) = 0,\\
        \phi(x,T) = 0.
    \end{cases}
\end{align}

\noindent \textbf{For $f(u) = \alpha u$ for some $\alpha >0$}, the system of dual equations are two linear transport equations, $u \,(\phi)$ runs forward (backward) in time:
\begin{align}
\label{eqn:linear_sys}
    \begin{cases}
     &    u_t + \alpha u_x  -\partial_{x} \left(\gamma(x) \partial_x u\left(x,t\right) \right)= 0,\\
& u(x,0) = u_0(x),\\
& \phi_t + \alpha \phi_x  + \partial_{x} \left(\gamma(x) \partial_x \phi\left(x,t\right) \right)= 0,\\
& \phi(x,T) = 0.
    \end{cases}
\end{align}

\noindent \textbf{For $f(u) = \alpha u^2 $ for some $\alpha >0$}, the system of dual equations are  as follows, where $u $ solves a quadratic conservation law with initial condition specified, while $\phi$ satisfies a backward transport equation that couples $u$ in the transportation term:
\begin{align}
\label{eqn:quad_conv_sys}
    \begin{cases}
     &    u_t + \alpha \partial_x (u^2) - \partial_{x} \left(\gamma(x) \partial_x u\left(x,t\right) \right)= 0,\\
& u(x,0) = u_0(x),\\
& \phi_t + (2 \alpha u) \phi_x + \partial_{x} \left(\gamma(x) \partial_x \phi\left(x,t\right) \right)= 0,\\
& \phi(x,T) = 0.
    \end{cases}
\end{align}
\subsection{The primal-dual optimization method}
The primal-dual hybrid gradient (PDHG) method is a first-order method that solves constrained and non-differentiable optimization problems with a saddle point structure.
Given some Hilbert spaces $\mathcal{X,H}$,  convex function $f, g$,
and a linear map $A: \mathcal{X} \rightarrow \mathcal{H}$, denote $g^*(q) = \sup_{p} \langle p, q \rangle - g(p)$, the convex conjugate of function $g$. 
The saddle point problem takes the following form
\begin{equation*}
    \min_p \max_q \quad h(p) +  \langle Ap, q \rangle - g^*(q),
\end{equation*}
which is equivalent to the following minimization problem 
\begin{align*}
     \min_p  \quad h(p) +   g(Ap).
\end{align*}
The PDHG methods take proximal gradient descent (ascent) steps on variable $p \,( q )$ alternatively. Set $\tau_p,\tau_q >0$ as stepsizes, $(p^0,q^0)$ as the initial guess. The details of the iterations at $n$-step are as follows
\begin{align*}
\begin{cases}
    p^n & = \argmin \; h(p) +  \langle p, A^T \tilde{q}^{n-1} \rangle + \dfrac{1}{2 \tau_p} \| p-p^{n-1}\|^2_{L^2},\\
       q^n & = \argmax\;  \langle Ap^n, q \rangle + g(q) -\dfrac{1}{2 \tau_q} \| q-q^{n-1}\|^2_{L^2},\\
       \tilde{q}^{n} & = 2 q^n - q^{n-1}.
\end{cases}
\end{align*}
Here, $\|u\|^2_{L^2} =  \langle u,u \rangle _{L^2} = \int_0^T \int_{\Omega} u^2 dx dt$.
The algorithm is not sensitive to the initial guess and converges globally if the stepsize satisfies the following:
\begin{align*}
    \tau_p \tau_q < \dfrac{1}{\|A^T A\|}.
\end{align*}

Extensions and generalizations of PDHG methods have been carefully investigated.
In \cite{JacobsLegerLiOsher2018_solvinga}, the General-proximal Primal-Dual Hybrid Gradient (G-prox PDHG) is proposed with a focus on linear constraints induced by differential equations.
With proper choices of norms for the proximal steps, the optimization algorithm enjoys the property that the optimization stepsizes are free from the grid size.  In \cite{valkonen2014primal}, the operator $A$ is extended to non-linear. The method adopts the linear approximation of $A$ and guarantees the local convergence when some technical conditions are satisfied. 

We adapt ideas from the General-prox PDHG and discuss the linear transport equation and quadratic conservation laws about the choice of norms. The latter one also integrates the linearization method from \cite{valkonen2014primal}.

\noindent \textbf{Linear transport equation with $f(u) = \alpha u$.}
The min-max problem \eqref{eqn:min_max} has its saddle point satisfying Equation \eqref{eqn:linear_sys}. We omit the initial-terminal conditions and use the following notation:
\begin{align*}
   & p  = u,\quad q  = \phi,\\
 &    A(u) = \left(\partial_t + \alpha \partial_x\right)u -\partial_{x} \left(\gamma \partial_x u \right),\\
 &   A^T(\phi)  = -\left(\partial_t + \alpha \partial_x\right)\phi -\partial_{x} \left(\gamma \partial_x \phi \right),\\
 &   h = 0,\quad g(Ap) =\begin{cases}
        0,\quad \text{if} \;\, Ap = 0\\
        + \infty, \quad \text{else}.
    \end{cases} 
\end{align*}
The PDHG has a convergence rate rate $O(1/N)$ in finite dimensions, where the constant factor is proportional to $\frac{1}{\tau_p \tau_q }$.
In our case, $A$ contains differential operators. Hence it is easy to see that as we refine the mesh grid, smaller $\tau_p,\tau_q$ are needed, leading to a slower convergence.
We approximate the operator $ A^T A$ with
\begin{align}\label{eq:K_linear0}
    K = -\partial_{tt} - \alpha^2 \partial_{xx}  + \hat{\gamma}^2 \partial_{xxxx}.
\end{align}
 Note here $\hat{\gamma}$ is a constant approximation of the coefficient function $\gamma(x)$. Then we introduce the norm $\| \cdot \|_{\mathcal{H}}$:
\begin{align*}
    \| u \|^2_{\mathcal{H}} & = \|\partial_t u\|_{L^2}^2  +\alpha^2 \|\partial_x u\|_{L^2}^2   + \hat{\gamma}^2 \|\partial_{xx} u\|_{L^2}^2 \\
    & = \langle u,Ku \rangle_{L^2} + \int_{\Omega} u u_t dx|_{t = 0},
\end{align*}
where the last inequality is obtained via integration by parts with the periodic boundary condition in space.
We modify the algorithm accordingly for problem \eqref{eqn:min_max} with the norm $  \| \cdot \|^2_{\mathcal{H}}$ for the G-prox PDHG. The $n$-th iteration is taken as follows:
\begin{align}\label{alg:linear_itr}
 \begin{cases}
    u^n & = \argmin\;  \langle u, A^T \tilde{\phi}^{n-1} \rangle + \dfrac{1}{2 \tau_u} \| u-u^{n-1}\|^2_{L^2},\\
       \phi^n & = \argmax\;  \langle Au^n, \phi \rangle - \dfrac{1}{2 \tau_{\phi}} \| \phi-\phi^{n-1}\|^2_{\mathcal{H}},\\
       \tilde{\phi}^{n} & = 2 \phi^n - \phi^{n-1}.
\end{cases}
\end{align}
Each update can be written explicitly,
\begin{align*}
    u^n & =  u^{n-1} - \tau_u \left(\left(\partial_t + \alpha \partial_x  \right) \tilde{\phi}^{n-1}      +\partial_{x} \left(\gamma \partial_x \tilde{\phi}^{n-1} \right)  \right),\\
       \phi^n & = \phi^{n-1} - \tau_{\phi} \left(\partial_{tt} + \alpha^2 \partial_{xx} -  \hat{\gamma}^2 \partial_{xxxx} \right)^{-1} \left(\left(\partial_t + \alpha \partial_x  \right)u^n -\partial_x (\gamma \partial_{x}u^n)\right).\\
\end{align*}
The computation for $\phi^n$ can be done by using Fast Fourier Transform (FFT).

\noindent \textbf{Quadratic conservation laws with $f(u) = \alpha u^2$.} The problem \eqref{eqn:min_max} loses the convex-concave structure, yet we can still modify the G-prox PDHG algorithm to compute the local saddle point.
We omit the initial-terminal condition as well as the Laplacian term for simplicity. The nonlinear constraint from the conservation law takes the form:
\begin{align*}
&     A(u) = \partial_t u +  \partial_x f(u),\\
&     \langle A(u),\phi \rangle = \int_0^T \int_{\Omega} \partial_t u +  \partial_x f(u) \,dx dt,
\end{align*}
and we use the following linear approximations
\begin{align*}
   &  \langle A(u),\phi \rangle \approx \langle A(\hat{u}),\phi \rangle + \langle u-\hat{u} , \nabla A(\hat{u})^T \phi \rangle, \\
    &   \nabla A(u)^T \phi =  \partial_t \phi + f'(u) \partial_x \phi.
\end{align*}
The last line is obtained by taking the first variation of the  functional above.
As for the choice of the norm for the General-proximal PDHG, rather than using the approximation of nonlinear operator $A^TA$, we only take into consideration the linear differential operators  and denote
\begin{align*}
&\hat{A} u = (\partial_t + c \partial_x) u,
\end{align*}
where the constant $c$ is chosen based on estimations from $f'(\cdot)$. 
Now the operator $K$ that approximates $\hat{A}^T \hat{A} $ is as follows:
\begin{align}\label{eq:K_linear}
 K &= -\partial_{tt} -c^2 \partial_{xx}.
\end{align}
The norm is defined as 
\begin{align*}
  \|u\|^2_{\hat{\mathcal{H}}} &=\|\partial_t u\|_{L^2}^2  +c^2 \|\partial_x u\|_{L^2}^2   \\ & =\langle u, K u\rangle_{L^2}+ \int_{\Omega} u u_t dx|_{t = 0},.
\end{align*}
At $n$-th step, the algorithm takes the following updates:

\begin{align}\label{alg:nonlinear_jtr}
 \begin{cases}
    u^n & = \argmin \; \langle u, \nabla A({u}^{n-1})^T \tilde{\phi}^{n-1} \rangle + \dfrac{1}{2 \tau_u} \| u-u^{n-1}\|^2_{L^2},\\
       \phi^n & = \argmax \;  \langle Au^n, \phi \rangle - \dfrac{1}{2 \tau_{\phi}} \| \phi-\phi^{n-1}\|^2_{\hat{\mathcal{H}}},\\
       \tilde{\phi}^{n} & = 2 \phi^n - \phi^{n-1}.
\end{cases}
\end{align}
The explicit updates for $(u,\phi)$ are as follows:
\begin{align*}
      u^n & =  u^{n-1} - \tau_u \left( \partial_t + f'(u^{n-1}) \partial_x \right) \tilde{\phi}^{n-1},\\
       \phi^n & = \phi^{n-1} - \tau_{\phi} \left(\partial_{tt} + c^2 \partial_{xx} \right)^{-1} \left( \partial_t u^n + \partial_x f(u^n)\right).\\
\end{align*}
\subsection{On the convergence of the algorithm}
The PDHG algorithm's convergence is based on the convex-concave structure of the min-max problem \eqref{eqn:min_max}. In our setup, we only have global convergence when the conservation laws are linear (i.e., $f(\cdot)$ is linear, $A$ is a linear operator). Indeed, we obtain geometric convergence for this case. 

\begin{proposition}
For a linear operator $A: \mathbb{R}^N \rightarrow \mathbb{R}^N $, denote $u^*$ as the solution for $Au = c$.  Solve the min-max problem $\min_u \max_{v} \langle Au-c,v\rangle$ using Algorithm \ref{alg:linear_itr}. $K: \mathbb{R}^N \rightarrow \mathbb{R}^N$ is the linear operator that induced the norm $\|\cdot\|_{\mathcal{H}}: \|v\| = \langle v, Kv\rangle_{L^2}$ in the proximal update of $\phi$. If the following inequalities are satisfied: for some $k>0$,
$\sigma\tau < k$,  $\| A^TK^{-1}A\| = k$, we have $\lim_{n \rightarrow\infty} Au^n -Au^* = 0$. 
\end{proposition}
\begin{proof}
Let $L(u,\phi) = \langle Au-c,\phi\rangle$.
At the $n$-th iteration, we have 
\begin{align*}
\phi^{n+1} & =\argmax \; L(u^n,\phi)-\frac{1}{2\tau} 
\|\phi-\phi^n\|^2_{\mathcal{H}},\\
u^{n+1}&= \argmin \; L(u,\tilde{\phi}^{n+1})+\frac{1}{2\sigma}\|u-u^n\|^2_{L^2}.
\end{align*}
Hence, we have 
\begin{align*}
\phi^{n+1}&=\phi^{n}  +\tau \left(K^{-1}(Au^n-c)\right),\\
\tilde{\phi}^{n+1}&=2\phi^{n+1}-\phi^n=\phi^n +2 \tau \left(K^{-1}(Au^n-c)\right),\\
u^{n+1}&=u^n-\sigma A^T\phi^n-2\sigma \tau A^TK^{-1}(Au^n-c),\\
\begin{pmatrix}
u^{n+1}\\ \phi^{n+1}\end{pmatrix}
& =M \begin{pmatrix}
u^{n}\\ \phi^{n}\end{pmatrix}
+\begin{pmatrix}
2\sigma \tau A^T K^{-1}c\\-\tau K^{-1}c\end{pmatrix},\\ 
\text{where }& M= \begin{pmatrix}&I-2 \sigma \tau A^T K^{-1}A  &-\sigma A^T \\ & \tau K^{-1}A & I
\end{pmatrix}.
\end{align*}
The eigenvalues of $M$ are complex conjugates
$\lambda= (1-\sigma \tau \nu) \pm i(\sigma \tau \nu-(\sigma \tau \nu)^2)^{1/2}$.
Here $i=(-1)^{1/2} $ and $\nu$ is an eigenvalue of $A^T K^{-1}A$.
So the eigenvalues of $M$ have 
\begin{align*}
    |\lambda|=(1-\sigma \tau \nu)^{1/2}.
\end{align*}

Since $\nu>0$ and $\sigma, \tau>0$ and their product  $\sigma \tau \nu<1$,we have a contraction map and the result converges geometrically.
\end{proof}

When operator $A$ is a discrete approximation of continuous differential operator and $K$ is chosen properly, one can apply the above proposition and have constant $k$ independent of the grid size.
We also provide numerical experiments in Section~\ref{sec:example_heat} to show that the convergence of the optimization problem for linear case  is independent of the grid size.
As for nonlinear conservation laws in general, we no longer have convex problems. Therefore, we no longer guarantee the convergence to the saddle point. For more details on the convergence of nonlinearly constrained optimization problems, we refer to \cite{valkonen2014primal,clason2017primal}. We leave the convergence study of this algorithm for nonlinear conservation laws in future work. 

During the computation, we check the convergence of the algorithm by checking the residuals of the system of the primal-dual equations. The residual is defined as the $L^2$ norm of the continuity equation of the $(u, \phi)$ as follows:
\begin{align}
    \text{Res}{(u,\phi)} &= \left[\|A(u)\|_{L^2},\|\nabla A(u)^T \phi\|_{L^2} \right].
\end{align}
At the saddle point solution$({u^*},{\phi^*})$ of the min-max problem \eqref{eqn:min_max}, we have $\text{Res}{({u^*},{\phi^*})} = 0.$
Here $\text{Res}({\hat{u},\hat{\phi}})$ measures the distance between the current solution  $(\hat{u},\hat{\phi})$ and the saddle point solution $({u^*},{\phi^*})$.
\section{A primal-dual approach for the discretization system}\label{sec:numerics}
In this section, we discuss in detail on using the primal-dual approach to solve implicit numerical PDE. We start with time and spatial discretizations. Next, we discuss the implementation of the primal-dual algorithm. At the end of this section, we discuss two extensions of the standard approach: mesh refinement and one-timestep updates.
\subsection{Time discretization}\label{sec:sub_time}
For the initial value problem on the time interval $[0,T]$, we use uniform mesh with length $h_t = \frac{T}{N_t}$, and denote $t_l = l h_t, u^l(x) = u(t_l,x)$ for $l = 0, ..., N_t$. 
By using the backward Euler scheme for the conservation laws, we obtain
\begin{align*}
    \begin{cases}
       & \dfrac{u^{l+1}(x) -u^l(x) }{h_t} + \partial_x \left(f\left(u^{l+1}(x)\right)\right) - \partial_{x}(\gamma \partial_{x} u^{l+1}(x)) = 0,\quad l = 0, ..., N_t-1,\\
& u^0(x) = u_0(x).       
    \end{cases}
\end{align*}
The min-max problem \eqref{eqn:min_max} with discretization in time is as follows:
\begin{align}\label{eq:min_max_Lt}
    \min_{u \in \{u^l(x)\}} \max_{\phi \in \{\phi^l (x)\}} L_t(u,\phi),
\end{align}
where 
\begin{align*}
    L_t(u,\phi) & = h_t \sum_{\substack{0 \leq l \leq N_t-1}} \int_{\Omega} \phi^l(x) \left(  \dfrac{u^{l+1}(x) -u^l(x) }{h_t} + \partial_x \left(f\left(u^{l+1}(x)\right)\right) - \partial_{x}\left(\gamma \partial_{x} u^{l+1}(x)\right)\right) dx \\
    & + \int_{\Omega}\phi^{N_t}(x)u^{N_t}(x) - \phi^0(x)u^0(x) dx\\
    & = - h_t \sum_{\substack{1 \leq l \leq N_t}} \int_{\Omega} u^l(x) \left(\dfrac{\phi^l(x) - \phi^{l-1}(x)}{h_t} + \partial_{x}\left(\gamma \partial_{x} \phi^{l-1}(x)\right) \right) +  \partial_x\phi^{l-1}(x) {f\left(u^{l}(x)\right)} dx \\
    & + \int_{\Omega}\phi^{N_t}(x)u^{N_t}(x) - \phi^0(x)u^0(x) dx.
\end{align*}

The last equality is obtained via summation by parts and integration by parts in the spatial domain with periodic boundary condition on $\Omega$. By taking the first order variational derivative with respect to $u^l(x)$ for the last two lines, we obtain a forward Euler scheme for the dual equation of $\phi$
\begin{align*}
   \begin{cases}
      \dfrac{\phi^l(x) - \phi^{l-1}(x)}{h_t} +  \partial_x\phi^{l-1}(x) f'(u^l(x))  + \partial_{x}\left(\gamma \partial_{x} \phi^{l-1}(x)\right) = 0, ,\quad l = 1, ..., N_t,\\
      \phi^{N_t}(x) = 0. 
   \end{cases} 
\end{align*}

This is also an implicit scheme since we have $\phi$ with given terminal time condition. It is also straightforward to use other higher order implicit schemes, for instance, the backward differentiation formulas (BDFs). Specifically, one can first discretize the Equation \eqref{eq:conservation_law_gamma} using BDF2. By summation by parts over the time index, integration by parts in the spatial domain, and taking first order variational derivative, we automatically obtain the BDF2 for the dual equation $\phi$.
\subsection{Spatial discretization}
For the primal-dual framework, the choice of spatial discretization is quite flexible. Here we present both the finite difference scheme and the discontinuous Galerkin methods. For simplicity, we consider $\Omega = [0,b]$ with periodic boundary condition. We use a uniform mesh with $h_x = \frac{b}{N_x}$, for $ N_x>0$ and $x_j = j h_x, j = 0, ..., N_x-1$. We also apply the backward Euler formulation from the  Section \ref{sec:sub_time}.
\subsubsection{Finite difference scheme for the heat equation}
Consider a heat equation
\begin{align}\label{eq:heat_eqn}
    \begin{cases}
     &   u_t - \partial_{x}\left(\gamma \partial_{x} u\right)= 0,\\
     & u(x,0) = u_0(x),
    \end{cases}
\end{align}
with $\gamma(x)>0$. We use standard central difference scheme and denote $Lap(v)_j \approx\partial_{x}\left(\gamma \partial_{x} v\right)$ as its approximations, which is defined as follows:
\begin{align*}
   Lap(v)_j := \dfrac{\gamma_{j+\frac{1}{2}}\left( v_{j+1} -v_j\right) - \gamma_{j -\frac{1}{2}} \left(v_j -v_{j-1}\right)}{{h_x}^2}. 
\end{align*}
Therefore, the discretized min-max problem \eqref{eqn:min_max} takes the following finite difference scheme formulation:
\begin{align}
\label{eq:min_max_dis_heat}
  \min_{u \in \{u^l_j\}} \max_{\phi \in \{\phi^l_j\}} L(u,\phi),
\end{align}
where
\begin{align*}
     L(u,\phi) & =h_t h_x\sum_{\substack{0 \leq l \leq N_t-1\\  1\leq j \leq N_x}}  \phi^l_j \left(  \dfrac{u^{l+1}_j -u^l_j }{h_t} -   Lap(u^{l+1})_j \right) \\
     & = - h_t h_x  \sum_{\substack{1 \leq l \leq N_t\\  1\leq j \leq N_x}}    u^l_j \left(\dfrac{\phi^l_j - \phi^{l-1}_j}{h_t} +  Lap(\phi^{l-1})_j \right) 
 + h_x\sum_{\substack{1\leq j \leq N_x}}\left( \phi^{N_t}_j u^{N_t}_j - \phi^0_j u^0_j \right).
\end{align*}
The saddle point of the min-max problem \eqref{eq:min_max_dis_heat} is:
\begin{align}
    \begin{cases}
       \dfrac{u^{l+1}_j -u^l_j }{h_t} -   Lap(u^{l+1})_j = 0, \quad 0 \leq l \leq N_t-1,\; 1\leq j \leq N_x,\\
       u^0_j = u_0(x_j), \quad 1\leq j \leq N_x,\\
      \dfrac{\phi^l_j - \phi^{l-1}_j}{h_t} +   Lap(\phi^{l-1})_j =0, \quad 1 \leq l \leq N_t,\; 1\leq j \leq N_x,\\
      \phi^{N_t}_j = 0, \quad 1\leq j \leq N_x.
    \end{cases}
\end{align}
\subsubsection{Discontinuous Galerkin methods for linear transport equations}\label{subsec:dg_linear}
We first take partition of the interval $(0,b)$ into $N$ cells, then we have 
\begin{align*}
   & 0 = x_{\frac{1}{2}}<  x_{\frac{3}{2}} <  ... <  x_{N + \frac{1}{2}} = b,\\
    & h = h_j =  x_{j + \frac{1}{2}} -  x_{j- \frac{1}{2}}, I_j = \left( x_{j - \frac{1}{2}} , x_{j +  \frac{1}{2}} \right).
\end{align*}
Then we define the finite element space $V_h^k$ in the follows:
\begin{align*}
    V_h^k := \left\{ v\in L^2([0,b]): v|_{j_j} \in P^k(I_j), i = 1,...,N \right\}.
\end{align*}
For discontinuous Galerkin method, we seek for  $u_h(\cdot, t) \in V_h$, such that for any $v\in V_h$ we have
\begin{equation*}
\label{e2}
\frac{d}{dt} \int_{j_j} u_h(x,t) v(x) dx - \int_{j_j} f(u_h(x,t)) v'(x) dx + \hat{f}_{j+1/2} v(x_{j+1/2}^-)
- \hat{f}_{j-1/2} v(x_{j-1/2}^+) = 0,
\end{equation*}
where the numerical flux $\hat{f}_{j+1/2}$ can be any monotone flux:
$$
\hat{f}_{j+1/2} = \hat{f}( u_h(x_{j+1/2}^-,t), u_h(x_{j+1/2}^+,t))
$$
with $\hat{f}$ being monotonically increasing (non-decreasing) for the first argument and
monotonically decreasing (non-increasing) for the second argument.
For the simple case $f'(u) \geq 0$ for the initial condition, we can use the upwinding flux $\hat{f}(u^-,u^+) = f(u^-)$.
For each cell, we take two points $x_{j,1} = x_{j-\frac{1}{4}}, x_{j,2} = x_{j+\frac{1}{4}}$. Then we take the basis for $V_h$ to consist of the following $2N$ piecewise linear functions $ \{ \varphi_{j,\ell} (x) : \ell = 1, 2; i=1, ..., N \}$,
which satisfy
\begin{align*}
  \varphi_{j,1} (x_{j,1}) = 1, \quad \varphi_{j,1} (x_{j,2}) = 0, \quad \varphi_{j,1} (x)=0 \mbox { if } x \notin I_j,\\
  \varphi_{j,2} (x_{j,1}) = 0, \quad \varphi_{j,2} (x_{j,2}) = 1, \quad \varphi_{j,2} (x)=0 \mbox { if } x \notin I_j.
\end{align*}

If we denote
$
u_j = \left(
\begin{array}{c}
u_{j,1} \\ u_{j,2}
\end{array} \right)
$ for the linear transport equation with $f(u) = \alpha u, \alpha >0, \gamma = 0$, we obtain the following
\begin{equation}
\label{eqn:dg_linear}
\begin{aligned}
    \frac{d}{dt}u_j(t) = A_1 u_j + A_2 u_{j-1},\\
\end{aligned}
\end{equation}
where 
\begin{equation*}
A_1 =  \frac{\alpha}{h_x}  \begin{pmatrix}
-\frac{7}{4} & -\frac{3}{4}\\
\frac{11}{4} & -\frac{9}{4}
   \end{pmatrix},\quad
A_2=    \frac{\alpha}{h_x} \begin{pmatrix}
-\frac{5}{4} & \frac{15}{4}\\
\frac{1}{4} & -\frac{3}{4}
   \end{pmatrix}.
\end{equation*}
Combining Equation \eqref{eqn:dg_linear} with time discretization \eqref{eq:min_max_Lt}, we obtained the discretized min-max problem \eqref{eqn:min_max} :
\begin{align}
\label{eq:min_max_dis_transport}
  \min_{u \in  {\{u^l_j\}} }\max_{\phi \in \{\phi^l_j\}} L(u,\phi),
\end{align}
where
\begin{align*}
     L(u,\phi) & =h_t h_x\sum_{\substack{0 \leq l \leq N_t-1\\  1\leq j \leq N_x}}   \langle  \phi^l_j,   \dfrac{u^{l+1}_j -u^l_j }{h_t} - A_1u^{l+1}_j - A_2 u^{l+1}_{j-1}  \rangle, \\
     & = - h_t h_x  \sum_{\substack{1 \leq l \leq N_t\\  1\leq j \leq N_x}}     \langle  u^l_j, \dfrac{\phi^l_j - \phi^{l-1}_j}{h_t}  + A_1^T \phi^{l-1}_j  + A_2^T \phi^{l-1}_{j+1} \rangle
 + h_x\sum_{\substack{1\leq j \leq N_x}}\left(\langle \phi^{N_t}_j u^{N_t}_j\rangle - \langle\phi^0_j u^0_j \rangle\right).
\end{align*}
By taking the first-order optimality condition, we obtain the saddle point of the min-max problem \eqref{eq:min_max_dis_transport} as follows:
\begin{align}
    \begin{cases}
        \dfrac{u^{l+1}_j -u^l_j }{h_t} - A_1u^{l+1}_j - A_2 u^{l+1}_{j-1} = \left(\begin{array}{c}
0 \\ 0
\end{array}\right),\quad 0 \leq l \leq N_t-1,\; 1\leq j \leq N_x,\\
       u^0_j = \left(\begin{array}{c}
u_0(x_{j-\frac{1}{4}}) \\u_0(x_{j+\frac{1}{4}}) 
\end{array}\right), \quad 1\leq j \leq N_x,\\
       \dfrac{\phi^l_j - \phi^{l-1}_j}{h_t}  + A_1^T \phi^{l-1}_j  + A_2^T \phi^{l-1}_{j+1}  = \left(\begin{array}{c}
0 \\ 0
\end{array}\right),\quad 1 \leq l \leq N_t,\; 1\leq j \leq N_x,\\
      \phi^{N_t}_j =\left(\begin{array}{c}
0 \\ 0
\end{array}\right), \quad 1\leq j \leq N_x.
    \end{cases}
\end{align}

\subsubsection{Discontinuous Galerkin methods for quadratic conservation laws}
\label{subsec:dg_quadratic}
In this part, we discuss on numerical approximations of DG methods for quadratic conservation laws where $f(u) = \alpha u^2 + \beta u, \alpha>0,\beta>0$. We follow the same setup as in Section \ref{subsec:dg_linear} and use linear basis functions and upwind flux.
The DG scheme can be written in the following form:
\begin{equation}
\label{eqn:quad_foward_DG}
\frac{d}{dt} u_j(t) = \left( \begin{array}{c} u_{j-1}^T C_1 u_{j-1} \\ u_{j-1}^T C_2 u_{j-1} \end{array} \right)
+ \left( \begin{array}{c} u_{j}^T C_3 u_{j} \\ u_{j}^T C_4 u_{j} \end{array} \right)
+ \left( \begin{array}{c} u_{j+1}^T C_5 u_{j+1} \\ u_{j+1}^T C_6 u_{j+1} \end{array} \right)
+ C_7 u_{j-1} + C_8 u_{j} + C_9 u_{j+1}.
\end{equation}
The details of the matrices $C_i, i = 1,...,9$ can be find in Equation~\eqref{eqn:dg_coeff} in the Appendix.
For solving the quadratic conservation law \eqref{eqn:quad_conv_sys}, we solve the corresponding min-max problem \eqref{eqn:min_max} with the following discretized system:
\begin{align}
    \label{eq:min_max_quad_discretized}
      \min_{u \in  {\{u^l_j\}} }\max_{\phi \in \{\phi^l_j\}} L(u,\phi),
\end{align}
where the definition of $L$ can be found in the appendix Equation~\eqref{eqn:dis_L_quad_details}.

The corresponding saddle point of the discrete min-max problem \eqref{eq:min_max_quad_discretized} is presented in Equation~\eqref{eqn:dis_minmax_quad_details}.

\subsection{Primal-dual algorithm}
In section \ref{sec:continuous}, we have discussed how the generalized primal-dual algorithm solves the saddle point problem from an initial value problem. In this part, we solve the finite dimensional saddle point problem based on the numerical approximations of the PDEs. We use IVP \eqref{subsec:dg_linear} as an example and present details of the algorithm.

In the min-max problem \eqref{eqn:min_max}, we omit the initial-terminal condition when discussing the algorithm. In practice, there is no harm to introduce an extra dual variable $\lambda_j, j = 1,...,N,$ to impose the conditions. The discretized min-max system can be written as follows:
\begin{align}
    \min_{u} \max_{\phi,\lambda} L(u,\phi,\lambda),
\end{align}
where
\begin{align*}
         L(u,\phi,\lambda) & =h_t h_x\sum_{\substack{0 \leq l \leq N_t-1\\  1\leq j \leq N_x}}   \langle  \phi^l_j,   \dfrac{u^{l+1}_j -u^l_j }{h_t} - A_1u^{l+1}_j - A_2 u^{l+1}_{j-1}  \rangle  + h_x\sum_{\substack{1\leq j \leq N_x}} \langle \lambda_j , u^0_j - u_0(x_j)\rangle\\
     & = - h_t h_x  \sum_{\substack{1 \leq l \leq N_t\\  1\leq j \leq N_x}}     \langle  u^l_j, \dfrac{\phi^l_j - \phi^{l-1}_j}{h_t}  + A_1^T \phi^{l-1}_j  + A_2^T \phi^{l-1}_{j+1} \rangle\\
& \quad + h_x\sum_{\substack{1\leq j \leq N_x}}\left(\langle \phi^{N_t}_j, u^{N_t}_j\rangle - \langle \lambda_j -\phi^0_j, u^0_j \rangle -\langle \lambda_j , u_0(x_j)\rangle\right).
\end{align*}

The primal-dual approach for solving the linear transport equation is summarized in the following Algorithm \ref{alg:linear_equation}.

\begin{algorithm}
\caption{Solve linear transport equation via primal-dual approach.
}\label{alg:linear_equation}
\begin{flushleft}
\hspace*{\algorithmicindent} \textbf{Input:} $N_t,N >0$, initial guess $(u_j^l)^0,(\phi_j^l)^0,(\lambda_j)^0, l = 0,...,N_t, j = 1,...,N$; stepsizes $\tau_{z}, z\in \{u,\phi, \lambda \}$; residual tolerance $\epsilon$.\\
\hspace*{\algorithmicindent} \textbf{Output:} $(u_j^l)^n,(\phi_j^l)^n,(\lambda_j)^n$ for $n = n^*$.
\end{flushleft}
\begin{algorithmic}
    \While {iteration  $n<\mathcal{N}_{\text{max}}$ and Res$(u^n,\phi^n) > \epsilon$}
    \State{1. Primal updates \text{for} $\;l = 1,...,N_t, j = 1,...,N$:}
    \State{$(u^l_j)^n = (u^l_j)^{n-1} - \tau_u \left( \dfrac{(\tilde{\phi}_j^{l})^{n-1} - (\tilde{\phi}_j^{l-1})^{n-1}}{h_t}  + A_1^T (\tilde{\phi}_j^{l-1})^{n-1}  + A_2^T (\tilde{\phi}_{j+1}^{l-1})^{n-1}  \right)$, }
    \State{$(u^0_j)^n =  (u^0_j)^{n-1} -\tau_{u_0} ((\tilde{\phi}_j^{0})^{n-1} - (\tilde{\lambda}_j)^{n-1}  )$.}
    \State{2. Dual updates for $l = 0,...,N_t-1, j = 1,...,N,$:}
    \State{$(\phi^l_j)^n = (\phi^l_j)^{n-1} -\tau_{\phi} K_d^{-1} \left( \dfrac{(u^{l+1}_j)^n - (u^l_j)^n }{h_t} - A_1(u^{l+1}_j)^n  - A_2 (u^{l+1}_{j-1})^n \right),$}
    \State{$ (\lambda_j)^n =  (\lambda_j)^{n-1} + \tau_{\lambda}((u^0_j)^n - u_0(x_j)).$}
    \State{3. Extrapolation step for dual variables for all $l,j$:}
    \State{$(\tilde{\phi}_j^{l})^{n} = 2 ({\phi}_j^{l})^{n} - ({\phi}_j^{l})^{n-1},$}
    \State{$(\tilde{\lambda}_j)^{n} = 2 ({\lambda}_j)^{n} - ({\lambda}_j)^{n-1}.$  }
    \State{$n \gets n+1$}
    \EndWhile 
\end{algorithmic}
\end{algorithm}

The operator $K_d$ denotes the discrete operator of $K$ defined in Equation~\ref{eq:K_linear0}.
This step requires solving a Laplacian equation, which can be done efficiently via FFT.
As for the  initial guess of the optimization problem, without prior knowledge, we can use  $[(u_j^l)^0,(\phi_j^l)^0,(\lambda_j)^0] = [u_0(x_j),0,0], l = 0,...,N_t, j = 1,...,N$. The updates for $[u_j^l,\phi_j^l,\lambda_j$ from the Algorithm \ref{alg:linear_equation} are independent of each other in terms of index $(j,l)$. Hence, we can update  variables simultaneously for all space-time, which means the updates are highly parallelizable. The stopping criteria needs to be understood as each entry of Res$(u^n,\phi^n)$ satisfies $ \leq \epsilon$. The choice of $\epsilon$ varies, depending on the usage of the solution.

When the PDE is nonlinear, for instance, $f(u) = \alpha u^2 + \beta u$,  we choose  $K_d$ as the discrete operator of $(-\partial_{tt} + c^2 \partial_{xx})$ as discussed in Equation \ref{eq:K_linear}, where $c$ is an estimation of $f'(u)$ based on the initial condition. In practice, we can set  $c = 1$. The stepsizes of the optimization algorithm then need to be adjusted accordingly. 
\subsection{Mesh refinement in time approach}\label{subsec:mesh_refine}
We have observed that due to the lack of convexity, the primal-dual algorithm is not very robust and sometimes sensitive to the optimization stepsizes and initial data. We discuss one stabilizing technique here. 
The idea is to do mesh refinement in time. Since large $h_t$ brings in more numerical viscosity, it is expected that solving the problem on a coarse-mesh grid is easier and gives a smooth  numerical approximation of the solution.
We can do  some simple interpolation with the solution on coarse-mesh  and use it as an initial guess for fine-mesh  case. 
This extension is summarized in Algorithm~\ref{alg:mesh_refinement}, using vanilla version Algorithm~\ref{alg:linear_equation} repeatedly while refining the mesh in time. 
\begin{algorithm}
\caption{Generalize the primal-dual approach with mesh refinement.
}\label{alg:mesh_refinement}
\begin{flushleft}
   \hspace*{\algorithmicindent} \textbf{Input:} $N_t = 2^{m_0},N >0$, initial guess $(u_j^l)^0,(\phi_j^l)^0,(\lambda_j)^0, l = 0,1, j = 1,...,N$; stepsizes $\tau_{z}, z\in \{u,\phi, \lambda \}$; residual tolerance $\epsilon$, $N_{t_0} = 1$.\\
    \hspace*{\algorithmicindent} \textbf{Output:} $(u_j^l)^n,(\phi_j^l)^n,(\lambda_j)^n$ for $n = n^*$ , $l = 0,...,N_t, j = 1,...,N$.
\end{flushleft}
\begin{algorithmic}
    \While{$N_{t_0} \leq 2^{m_0}$}
    \State{ Solve the IVP on mesh $N \times N_{t_0}$ in the space-time domain using Algorithm \ref{alg:linear_equation}, obtain solution  $(u,\phi,\lambda)^*_{N \times N_{t_0} }$.}
\State{Interpolate the solution on a finer mesh $(u,\phi,\lambda)_{N \times 2N_{t_0}}$. }
\State{Use the above solution as an initial guess for discrete problem on mesh $N \times 2N_{t_0}$ }.
    \State{$N_{t_0} \gets 2 N_{t_0}$.}
    \EndWhile 
\end{algorithmic}
\end{algorithm}

\subsection{An iterative approach for solving implicit scheme for one-timestep}\label{subsec:one-step} 
In previous sections, we use the primal-dual framework to solve the initial value problem (IVP) defined on the interval $[0,T]$. It is natural to see that this framework could also be applied to solve the IVP on a smaller time interval $[0, h_t]$. Therefore, we can adapt the primal-dual approach as an iterative solver for solving one-timestep,  i.e., $[t, t+h_t]$, of the implicit discrete system. 
Unlike the vanilla primal-dual approach where the solutions are updated simultaneously over all space-time, the one-timestep approach only have iterations for solutions at time $t=t_l$ when solving over $[t_l-h_t, t_l]$. When the numerical solution is found at $t = t_l$ (the stopping criteria for the residuals is satisfied), the algorithm proceeds to $[t_l,t_l+ht]$, solving solutions at $t = t_{l+1}$.
For consistency, the residual is defined as Res$(u,\phi) = \left[\|A(u)\|_{L^2(\Omega \times \{t\})},\|\nabla A(u)^T \phi\|_{L^2(\Omega \times \{t\})} \right].$

\begin{algorithm}[H]
\caption{Use the primal-dual approach as an iterative solver.
}\label{alg:one-step}
\begin{flushleft}
    \hspace*{\algorithmicindent} \textbf{Input:} $N_t, N >0$; stepsizes $\tau_{z}, z\in \{u,\phi, \lambda \}$; residual tolerance $\epsilon$, $l=0$\\
    \hspace*{\algorithmicindent} \textbf{Output:} $(u_j^l)^*$ for $l = 0,...,N_t, j = 1,...,N$.
\end{flushleft}
\begin{algorithmic}
\State{$(u_j^0)^* = u_0(x_j)$}
    \While{$l < N_t$}
    \State{Set initial guess as $(u_j^k)^0 = (u_j^{l})^*$ for $k = l, l+1$.}
    \State{Apply Algorithm \ref{alg:linear_equation} to solve IVP on $[l h_t, (l+1)h_t]$ with initial data given as $u(l h_t,x_j) = (u_j^l)^*$. The mesh is discretized as $N \times 2$ in space-time domain.}
    \State{Obtain solution $(u_j^k)^{n^*}, k = l, l+1$.}
\State{$(u_j^l)^* = (u_j^l)^{n^*}$. }

    \State{$l \gets l+1$.}
    \EndWhile 
\end{algorithmic}
\end{algorithm}

\section{Numerical Examples}\label{sec:example}
In this section, we present numerical experiments on three types of equations: the heat equation, the linear transport equation, and the quadratic conservation laws. 
We use the conservation law example to show that our framework not only works for linear equations, but also can be generalized to accommodate nonlinear conservation laws. Besides, we test on various initial data: smooth or with discontinuities.
\subsection{Variable coefficient type heat equation}\label{sec:example_heat}
In this example, we apply the primal-dual approach to solve the second order parabolic equation
\begin{align*}
\begin{cases}
     &\partial_t u = \partial_x \left(\gamma \partial_{x} u\right)  , \quad (x,t) \in [0,1]\times[0,0.1],\\
    & u(x,0) = \exp{(-64{(x-0.5)^2})}, \quad \text{for } x\in  [0,1],
\end{cases}
\end{align*}
where the coefficient function satisfies $\gamma(x) = 0.5 + 0.1 \sin{(2\pi x)}$.

We solve the initial value problem on different meshes using the backward Euler scheme in time and the finite difference scheme in space,
and record the number of iterations needed to guarantee the residuals satisfies the error tolerance for two values of $\epsilon$ in Table \ref{tab:1}.
The norm used in the proximal update for the dual variable $\phi$ is chosen as
\begin{align*}
    \|v\|_{\mathcal{H}}^2 = \|\partial_t v\|_{L^2}^2  + 0.6^2  \|\partial_{xx} v\|_{L^2}^2. 
\end{align*}
Note here we use the stepsizes $(\tau_{u},\tau_{\phi},\tau_{\lambda})=(0.8,0.8,0.99)$ in the primal-dual approach for above problems on different meshes.
Since we are using implicit schemes, the discretized timestep is no longer restricted by the CFL condition. Typically, for an explicit scheme we need to satisfy a restrictive condition, i.e.,  $\frac{h_t}{(h_{x})^2}  < \frac{1}{2\max_x \gamma(x)}$.
However, in the above examples we have $\frac{h_t}{(h_{x})^2} =\frac{128}{5}, \frac{256}{5},...,\frac{2048}{5}$.
We can see from Table~\ref{tab:1} that the convergence of the algorithm for this linear equation is independent of the grid size. 

\begin{table}[!htp]\centering
\caption{Number of iterations needed to achieve certain errors for the primal-dual approach for different mesh.}\label{tab:1}
\begin{tabular}{l|rrrrrr}\toprule
$N_t \times N$  & $64 \times 16$ & $128 \times 32$ & $256 \times 64$ & $512 \times 128$ & $1028 \times 256$ \\ \hline
$\epsilon = 10^{-6} $ &90 &91 &92 &94 &102 \\
\hline
$\epsilon = 10^{-10} $ &144 &146 &147  &149 &163 \\
\bottomrule
\end{tabular}
\end{table}

\subsection{Linear transport equation}
In this example, we use piecewise linear discontinuous Galerkin method in the spatial domain and BDF2 in time to solve the following linear transport equation with $\alpha = 2$:

\begin{align*}
    \begin{cases}
       \partial_t  u  +  \alpha \partial_x u = 0, \quad \text{for } (x,t)\in[0,1]\times[0,0.5],\\
       u(x,0) = \sin(2 \pi x).
    \end{cases}
    \end{align*}

Thanks to the implicit scheme, we can choose use a time step much larger than the CFL-allowed time step for explicit schemes.
In this example, we have $\frac{\alpha h_t}{h_x} = 2$, while for piecewise linear DG method, the CFL number for explicit scheme is around $\frac{1}{3}$ (see  \cite{cockburn2001runge}).
We plot the numerical results in the Figure~\ref{fig:eglinear}.
On the right, we see the residuals Res$(u^n,\phi^n)$ decrease, while the distance between the approximated solution $u^n$ and the analytical solution $u^*$ reaches some positive value. This difference comes from the numerical approximation part, as the numerical scheme we are using are second order in space time. We also numerically verify that the  solutions we obtained from the primal-dual approach are indeed of second order accuracy. This is presented in Table~\ref{tab:2_linear_eqn_error}, where we refine the mesh both in space and in time and record the 
$L^2$ error $e_h =\|u_h-u_g\|_{L^2(\Omega\times [0,T])}$. Here $u_h$ denotes the numerical approximation over mesh $h_t,h_x$.

	\begin{figure}[htbp!]
		\includegraphics[width=0.950\textwidth]{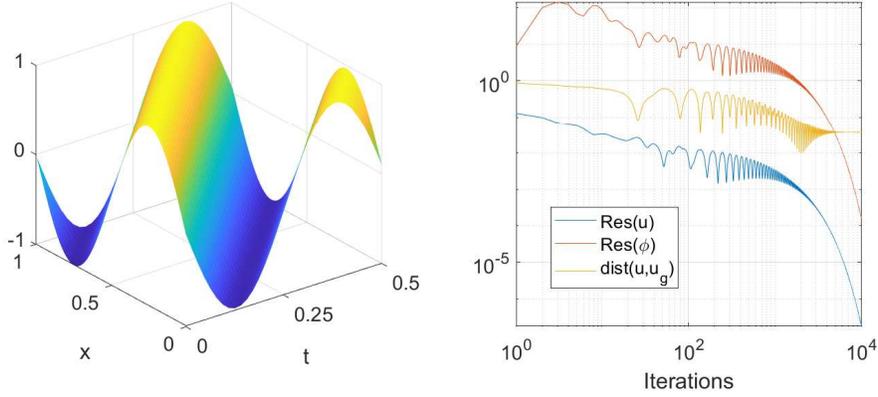}
		\caption{Left: numerical solution of the linear transport equation with smooth initial data, $N\times N_t = 64\times 32$. Right: change of residual Res$(u^n,\phi^n)$ and $\|u-u_g\|_{L^2}$ with respect to the primal-dual iteration, where $u_g$ is the analytical solution to the IVP. The optimization stepsizes are~$(\tau_{u},\tau_{\phi},\tau_{\lambda})=(3,0.1,0.99)$. }
		\label{fig:eglinear}	
	\end{figure}

Since the numerical scheme we are using is second order in space and time, we verify it numerically. We record the error of the solution below in Table~\ref{tab:2_linear_eqn_error}. From the table, we see that the solution is indeed of second order accuracy. 
\begin{table}[!htp]\centering
\caption{Errors of approximated solutions}\label{tab:2_linear_eqn_error}
\begin{tabular}{l|rrrr}\toprule
$h=h_x(h_t)$ & $2^{-5}$ & $2^{-6} $& $2^{-7}$ \\ \hline
$L^2$ error $e_h$ & $0.1296$ & $0.0372$ & $0.0096$ \\
$\log_2({e_{h}}/{e_{\frac{h}{2}})}$& $1.8007$ &$1.9542$ & - \\
\bottomrule
\end{tabular}
\end{table}

As for initial data with discontinuity, we apply the primal-dual approach to solve the following IVP:
\begin{align*}
    \begin{cases}
      \partial_t  u  +  \alpha \partial_x u = 0, \quad \text{for } (x,t)\in[0,1]\times[0,0.25],\\
      u(x,0) = \begin{cases}
          & 1, \quad \text{if} \quad0.25 \leq x \leq 0.75,\\
          & 0 \quad \text{else.}
      \end{cases}
    \end{cases}
\end{align*}
We use BDF2 for time discretization, and linear DG for space discretization. Here, $\alpha = 2, \frac{\alpha h_t}{h_x} = 1$. The numerical results are shown in Figure~\ref{fig:eglinear_dis}. We can see that our primal-dual approach solves IVP with discontinuous initial data successfully, where the solution have oscillations locally.
\begin{figure}[htbp!]
		\includegraphics[width=0.950\textwidth]{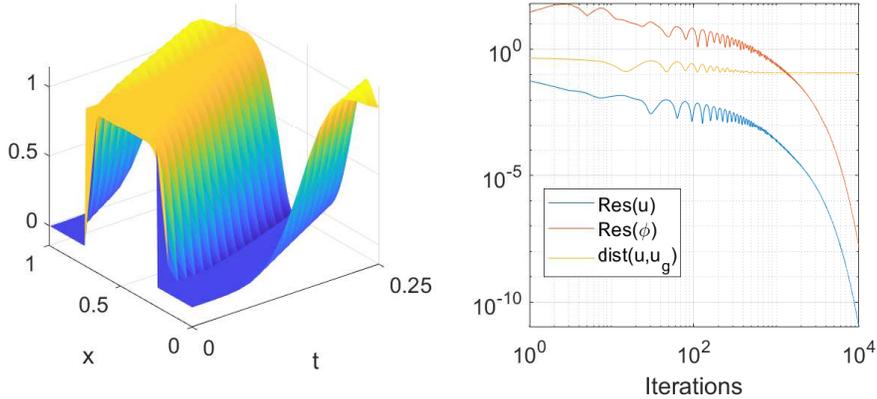}
		\caption{Left: numerical solution of the linear transport equation with discontinuous initial data, $N\times N_t = 64\times 32$. Right: change of residual Res$(u^n,\phi^n)$ and $\|u-u_g\|_{L^2}$ with respect to the primal-dual iteration, where $u_g$ is the analytical solution to the IVP. The optimization stepsizes are $(\tau_{u},\tau_{\phi},\tau_{\lambda})=(3,0.1,0.99)$.}
		\label{fig:eglinear_dis}	
\end{figure}
\subsection{Quadratic conservation law}
Consider a quadratic conservation law of the following form:
\begin{align}\label{eqn:num_quad_eg}
   & \partial_t u + \partial_x ( \alpha u^2 + \beta u) = 0, \quad \text{for } (x,t)\in[0,2]\times[0,1],\\
    &      u(x,0) = \begin{cases}
          & 0.1, \quad \text{if} \quad 1\leq x \leq 2,\\
          & 0.25 \quad \text{else.}
      \end{cases}
\end{align}
Here $\alpha = -1, \beta = 1$, we have the traffic equation.
Since the initial data of this given problem satisfies $f'(u) \geq 0$, we can use upwinding flux to approximate the term $\partial_x f(u)$. If the inequality does not hold, we can always introduce $s>0$ such that $f'(u) +s\geq 0$ for $u_0(x)$, and solve  equation $\partial_t v + \partial_x (f(v) + sv) = 0$ for $v(x,t) = u(x -st,t)$ via  change of variables. We can use upwinding flux for the later equation, and recover $u(x,t)$ correspondingly.
We use optimization stepsizes $(\tau_{u},\tau_{\phi},\tau_{\lambda})=(0.4,0.4,0.99)$ on the mesh $256\times 32$ with backward Euler in time.
The error tolerance $\epsilon = 10^{-3}$. It takes $n = 4573$ iterations to have the residual satisfies the stopping criteria.
The numerical results are presented in Figure~\ref{fig:eg_quad}, where we see the shock propagation (left) and development of rarefaction wave (right). 
\begin{figure}[htbp!]
		\includegraphics[width=0.950\textwidth]{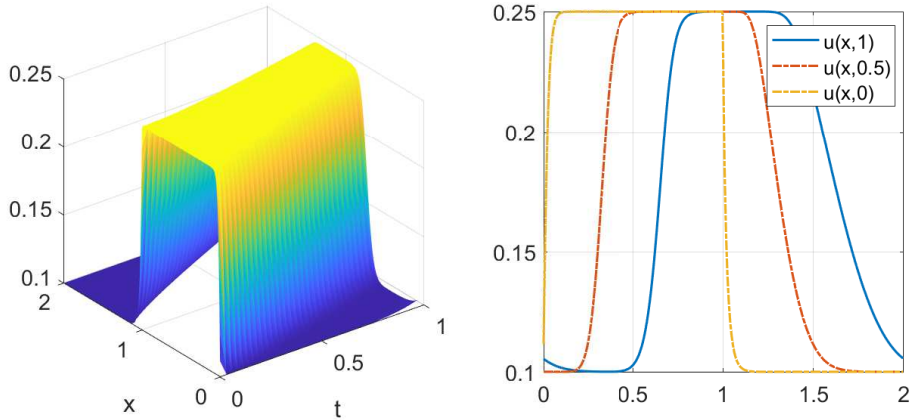}
		\caption{Left: numerical solution of the traffic equation with discontinuous initial data, $N\times N_t = 256\times 32$. Right: numerical solution at time $t = 0, 0.5, 1$. }
		\label{fig:eg_quad}	
\end{figure}

\subsubsection{A comparison of different variations of the primal-dual approach.}
For the nonlinear case, the choice of optimization stepsizes are quite sensitive for solving the non-convex saddle point problem. A good initial guess of the saddle point can always stabilize and boost the convergence of the primal-dual approach. Indeed, the variants of primal-dual approach we discussed in Sections \ref{subsec:mesh_refine} and \ref{subsec:one-step} can be understood as an improvement in the initial guess. For the mesh refinement approach, we take the solution from coarse grid as an estimation of the solution and use it the initial guess for the optimization. As for the one-timestep approach, since we are only considering the evolution of the initial data within a short period of time, a constant extension of function $u_{est}(x,t) =u_0(x)$ can be treated as a reasonable estimate of the solution.  
In this part, we compare  the vanilla primal-dual approach, the coarse-to-fine mesh variants, and the one-timestep method to compute the IVP~\ref{eqn:num_quad_eg}. We use BDF2 and backward Euler and set $\epsilon =  10^{-3}$.
We record the computation cost for each approach to achieve the residual tolerance.
Specifically, for the mesh refinement approach, we solve the IVP while refining the mesh with $N\times N_t = 256\times 2^k,$ \text{for} $ k = 3,4,5,$ in order. For $k=3,4$, we apply fixed number of iteration $n=1000$.
As for the one-timestep method, each time we apply the primal-dual approach with error tolerance $\epsilon=  10^{-3}$.

\begin{table}[!htp]\centering
\caption{Comparison of different variants of the primal-dual approach}\label{tab:compare}
\begin{tabular}{|p{0.2\linewidth} |p{0.23\linewidth}|p{0.23\linewidth}|p{0.23\linewidth} |}\toprule
method & number of PDHG iteration & world time (if everything parallel) & equivalent number of operation (computation)  \\ \hline
primal-dual (vanilla) & $4573$ & $4573$ & $150909N$  \\ \hline
mesh-refinement &- & $5601$ & $144833N$ \\ \hline 
one-timestep & $8627$  & $8627$ & $17254 N$ \\
\bottomrule
\end{tabular}
\end{table}
In Table \ref{tab:compare}, the number of iterations refers to how many iterations each algorithm takes. Specifically, for the mesh-refinement approach, the numbers of iterations for optimization over different meshes are $[10^3,10^3,3601]$, $N\times N_t = 256\times 2^k,$ \text{for} $ k = 3,4,5$. As for the one-timestep method, the total number of iterations is the summation $\sum_{l} n^*_l$, where $n_l^*$ is the iterations to solve the forward problem over the time interval $[l h_t, (l+1)h_t]$. If all updates can be done in a completely parallel manner over all indices $(j,l)$, then the time, where the unit is the time to complete a pair of primal-dual updates, is summarized in the third column. We can see here that with sufficient computation power, the vanilla approach takes the least time. If we instead consider the computation cost, where we treat  updating a pair of primal-dual variables at certain index $(j,l)$ as a unit of operations, then the equivalent number of operations are calculated accordingly, taking into account of the mesh parameter $N,N_t$. For presentation purpose, we keep the notation of $N$, and plug in the value $N_t =32$. One can see that the actual one-timestep approach costs the least computation. Comparing with the vanilla approach, one-timestep avoids inaccurate information propagating along time. The solution is calculated precisely in a  forward propagation manner. 
\section{Summary}\label{sec:summary}
In this paper, we propose a novel primal-dual approach for implicitly solving conservation laws.  The proposed saddle problem depends on the dual variable linearly. This approach connects the first-order optimization with scalar conservation laws, which integrates the idea of primal-dual hybrid gradient algorithm and accommodate precondition naturally.
Despite drawbacks from potential non-convexity, this approach is stable and converges in the practice.
This framework is easy for implementation and highly parallelizable. Moreover, the structure of primal-dual approach has the flexibility that can be adapted to various high order scheme in the spatial domain, including finite difference schemes and DG. In future work, we shall study on the primal-dual hybrid gradient methods for computing high dimensional regularized conservation laws implicitly in time with TVD (Total variation diminishing), ENO (essentially non-oscillatory), WENO (weighted essentially non-oscillatory), and neural network discretizations.

\noindent \textbf{Acknowledgement}: The authors would like to thank Richard Tsai for the insightful discussion.

\appendix                                     

\section{Details on discontinuous Galerkin methods for quadratic conservation laws}
The coefficient matrices in Equation~\ref{eqn:quad_foward_DG} are defined as follows:
\begin{equation}
\begin{aligned}\label{eqn:dg_coeff}
    C_1 = \frac{5\alpha}{8h_x}\begin{pmatrix}
     1 &-3\\
     -3 &9
    \end{pmatrix}, 
       C_2 = \frac{-\alpha}{8 h_x}\begin{pmatrix}
    1 &-3\\
    -3 &9
    \end{pmatrix},\\
       C_3 = \frac{-\alpha}{8 h_x}\begin{pmatrix}
     13 &1\\
     1 &5
    \end{pmatrix},
C_4 = \frac{\alpha}{8 h_x}\begin{pmatrix}
     9 &13\\
     13 &-31
    \end{pmatrix},\\ 
C_7 = \frac{\beta}{4 h_x}\begin{pmatrix}
     -5 &15\\
     1 &-3
    \end{pmatrix},
C_8 = \frac{\beta}{4 h_x}\begin{pmatrix}
     -7 & -3\\
     11 &-9
    \end{pmatrix},\\ 
C_5 = C_6 = C_9 =\begin{pmatrix}
     0 & 0\\
     0 &0
    \end{pmatrix}.\end{aligned}
\end{equation}

\begin{equation} \label{eqn:dis_L_quad_details}
\begin{aligned}
     L(u,\phi) & =h_t h_x\sum_{\substack{0 \leq l \leq N_t-1\\  1\leq j \leq N_x}}   \langle  \phi^l_j,   \dfrac{u^{l+1}_j -u^l_j }{h_t} -\left( \begin{array}{c} (u^{l+1}_{j-1})^T C_1 u^{l+1}_{j-1} \\ (u^{l+1}_{j-1})^T C_2 u^{l+1}_{j-1} \end{array} \right)
- \left( \begin{array}{c} (u^{l+1}_{j})^T C_3 u_{j}^{l+1} \\ (u^{l+1}_{j})^T C_4 u^{l+1}_{j} \end{array} \right)
- \left( \begin{array}{c} (u^{l+1}_{j+1})^T C_5 u^{l+1}_{j+1} \\ (u^{l+1}_{j+1})^T C_6 u^{l+1}_{j+1} \end{array} \right)\\
& \;\hspace{3cm} - C_7 u^{l+1}_{j-1} - C_8 u^{l+1}_{j} - C_9 u^{l+1}_{j+1} \rangle. 
\end{aligned}
\end{equation}
Again by the first-order optimality condition, we arrive at the saddle point of the min-max problem \eqref{eq:min_max_quad_discretized} is as follows:
\begin{equation} \label{eqn:dis_minmax_quad_details}
\begin{aligned}
    \begin{cases}
       & \dfrac{u^{l+1}_j -u^l_j }{h_t} -\left( \begin{array}{c} (u^{l+1}_{j-1})^T C_1 u^{l+1}_{j-1} \\ (u^{l+1}_{j-1})^T C_2 u^{l+1}_{j-1} \end{array} \right)
- \left( \begin{array}{c} (u^{l+1}_{j})^T C_3 u^{l+1}_{j} \\ (u^{l+1}_{j})^T C_4 u^{l+1}_{j} \end{array} \right)
- \left( \begin{array}{c} (u^{l+1}_{j+1})^T C_5 u^{l+1}_{j+1} \\ (u^{l+1}_{j+1})^T C_6 u^{l+1}_{j+1} \end{array} \right)\\
& \; - C_7 u^{l+1}_{j-1} - C_8 u^{l+1}_{j} - C_9 u^{l+1}_{j+1} = \left(\begin{array}{c}
0 \\ 0
\end{array}\right),\quad 0 \leq l \leq N_t-1,\; 1\leq j \leq N_x,\\
  &     u^0_j = \left(\begin{array}{c}
u_0(x_{j-\frac{1}{4}}) \\u_0(x_{j+\frac{1}{4}}) 
\end{array}\right), \quad 1\leq j \leq N_x,\\
&        \dfrac{\phi^l_j - \phi^{l-1}_j}{h_t}  + \left(\left[C_1 u^l_{j+1},C_2 u^l_{j+1}\right] +  \left[C_3 u^l_j,C_4 u^l_j\right] + \left[C_5  u^l_{j-1},C_6 u^l_{j-1}\right]\right)\phi^{l-1}_j  + A_2^T \phi^{l-1}_{j+1}  \\
& + C_7 \phi^{l-1}_{j+1} + C_8 \phi^{l-1}_{j} + C_9 \phi^{l-1}_{j-1} = \left(\begin{array}{c}
0 \\ 0
\end{array}\right),\quad 1 \leq l \leq N_t,\; 1\leq j \leq N_x,\\
     & \phi^{N_t}_j =\left(\begin{array}{c}
0 \\ 0
\end{array}\right), \quad 1\leq j \leq N_x.
    \end{cases}
\end{aligned}
\end{equation}

\bibliographystyle{plain} 
\bibliography{ref1} 
\end{document}